\documentclass{birkart}

\usepackage[english]{babel}
\usepackage{amsmath,amsthm,amscd,amssymb,latexsym,enumerate}
\usepackage[T1]{fontenc}
\usepackage{cite}
\usepackage{humanist}
\usepackage[letterpaper,hmargin=1.35in,vmargin=1.30in]{geometry}

\DeclareMathOperator{\ca}{Cap}

\newcommand{\da}{\downarrow}
\newcommand{\up}{\uparrow}

\newcommand{\C}{\mathbb{C}}
\newcommand{\R}{\mathbb{R}}
\newcommand{\N}{\mathbb{N}}

\newcommand{\F}{\mathbb{F}}
\newcommand{\PP}{\mathbb{P}}

\newcommand{\DD}{{\mathbb D}}

\newcommand{\E}{\texthmin{E}}

\newcommand{\cm}{{\bf{x}}}

\newcommand{\RR}{\mathcal R}

\newcommand{\Om}{\Omega}
\newcommand{\om}{\omega}

\newcommand{\si}{\sigma}

\newcommand{\al}{\alpha}
\newcommand{\be}{\beta}
\newcommand{\Ga}{\Gamma}
\newcommand{\ga}{\gamma}

\newcommand{\de}{\delta}

\newcommand{\eps}{\varepsilon}

\newcommand{\ess}{\text{\rm{ess}}}

\renewcommand{\sc}{\text{\rm{sc}}}
\newcommand{\s}{\text{\rm{s}}}

\newcommand{\Arg}{\operatorname{Arg}}

\newcommand{\im}{\operatorname{Im}}
\newcommand{\re}{\operatorname{Re}}
\newcommand{\id}{\operatorname{id}}
\newcommand{\diam}{\text{\rm{diam}}}
\newcommand{\cl}{\operatorname{cl}}
\newcommand{\ran}{\operatorname{Ran}}
\DeclareMathOperator{\supp}{supp}

\newtheorem{theorem}{Theorem}
\newtheorem{lem}{Lemma}
\newtheorem{prop}{Proposition}
\newtheorem*{cor}{Corollary}
\newtheorem*{thm}{Theorem}
\theoremstyle{definition}
\newtheorem*{definition}{Definition}

\theoremstyle{remark}
\newtheorem*{remark}{Remark}
\newtheorem*{example}{Example}

\numberwithin{equation}{section}

\begin{document}

\title
{Szeg\H{o}'s theorem on Parreau--Widom sets}

\author
{Jacob S. Christiansen}

\address{
Department of Mathematical Sciences \\
University of Copenhagen \\
Universitetsparken 5 \\
2100 Copenhagen \\
Denmark}

\email{stordal@math.ku.dk}

\thanks{The author was supported by a Steno Research Grant (09-064947) from the Danish Research Council for Nature and Universe.}




\begin{abstract}
In this paper, we generalize Szeg\H{o}'s theorem for orthogonal polynomials on the real line to infinite gap sets of Parreau--Widom type. This notion includes Cantor sets of positive measure. The Szeg\H{o} condition involves the equilibrium measure which is shown to be absolutely continuous. Our approach builds on a canonical factorization of the $M$-function and the covering space formalism of Sodin--Yuditskii.
\end{abstract}

\maketitle

\section{Introduction}
Let $d\nu=w(\theta)\frac{d\theta}{2\pi}+d\nu_\s$ be a finite positive measure on the unit circle $\partial\DD$, with $d\nu_\s$ singular to $d\theta$. A classical result of Szeg\H{o} \cite{Szego} reads
\begin{equation}
\label{Sz thm}
\inf_{p\in\PP}\biggl\{\int\vert 1-p\vert^2\, d\nu\biggr\}=
\exp\biggl\{\int_0^{2\pi}\log w(\theta)\frac{d\theta}{2\pi}\biggr\},
\end{equation}
where $\PP$ is the set of polynomials vanishing at zero. So the infimum on the left-hand side
is $>0$ if and only if the integral on the right-hand side, also known as the \emph{Szeg\H{o} integral}, is convergent (i.e., $>-\infty$). Strictly speaking, Szeg\H{o} only considered absolutely continuous measures but one can allow for a singular part too.

There are many equivalent forms of Szeg\H{o}'s theorem, see, e.g., \cite[Chap.~2]{MR2105088}, and it is bound up with asymptotics of Toeplitz determinants. From the point of view of orthogonal polynomials, perhaps the most suitable formulation is due to Verblunsky \cite{MR007}. It re\-places the left-hand side in \eqref{Sz thm} by
\begin{equation}
\label{product} \notag
\prod_{n=0}^\infty\bigl(1-\vert\al_n\vert^2\bigr),
\end{equation}
where $\{\al_n\}_{n=0}^\infty$ are the recurrence coefficients of the associated monic orthogonal polynomials.

Without any problems, one can carry over the result of Szeg\H{o} to measures supported on an interval of the real line. Let $d\mu=f(t)dt+d\mu_\s$ be a probability measure supported on $[-2,2]$ and let $\{a_n, b_n\}_{n=1}^\infty$ be the recurrence coefficients of the associated orthonormal polynomials. Then
\begin{equation}
\label{[-2,2]}
\inf_{n\in\N}\, (a_1\cdots a_n)>0
\quad\iff\quad
\int_{-2}^2\frac{\log f(t)}{\sqrt{4-t^2}}\,dt>-\infty.
\end{equation}
It is possible to allow for point masses of $d\mu_\s$ outside $[-2,2]$ as long as the mass points $\{x_k\}$ satisfy the condition
\begin{equation}
\label{sqrt} \notag
\sum_k \sqrt{(x_k+2)(x_k-2)}<\infty,
\end{equation}
see \cite{MR1844996, MR2020274} for details. If infinite in number, the $x_k$'s thus have to accumulate sufficiently fast at the endpoints $\pm 2$.

The aim of this paper is to establish a version of Szeg\H{o}'s theorem on sets in $\R$ much more general than an interval, namely what we shall call Parreau--Widom sets. The notion of such sets will be introduced in Sect.~\ref{secPW} and we shall only give a brief description here.
Among all regular compact subsets of $\R$, the Parreau--Widom condition \eqref{PW} singles out those for which the values of the Green's function at critical points are summable. In particular, it allows for a great number of sets with infinitely many components.

Our main result, Thm.~\ref{main thm} (in Sect.~\ref{proof}), goes beyond the recent monograph \cite{Simon} of Simon as to generalizing the Szeg\H{o}--Shohat--Nevai theorem, using the language of \cite{Simon}. The situation of `finite gap sets' was studied in \cite{MR2659747, CSZ2}, inspired by Widom's famous paper \cite{MR0239059} and the landmark paper \cite{MR1981915} of Peherstorfer--Yuditskii. For a finite gap set $\frak{e}$, the product in \eqref{[-2,2]} has to be replaced by
\begin{equation}
\label{cap}
\frac{a_1\cdots a_n}{\ca(\frak{e})^n},
\end{equation}
where $\ca(\frak{e})$ is the logarithmic capacity of $\frak{e}$. Moreover, the Szeg\H{o} condition takes the form
\begin{equation}
\label{Sz int}
\int_\frak{e}\log f(t)\, d\mu_\frak{e}(t)>-\infty,
\end{equation}
where $d\mu_\frak{e}$ is the equilibrium measure of $\frak{e}$. For comparison, $d\theta/2\pi$ is the equilibrium measure of $\partial\DD$ and $d\mu_{[a,b]}$ is a scaled arcsine distribution.
While \cite{MR0239059} uses multiple-valued functions, \cite{MR1981915} makes heavy use of the covering space formalism introduced by Sodin--Yuditskii \cite{MR1674798}. The present paper also builds on \cite{MR1674798}, as do \cite{MR2659747, CSZ2}, and we give some background on uniformization theory in Sect.~\ref{secPW}.

The core of \cite{MR1981915, PY} is to establish Szeg\H{o} asymptotics of orthogonal polynomials on so-called homogeneous sets in $\R$. This is done under the Szeg\H{o} condition and a Blaschke condition (similar to \eqref{Blaschke}) by comparing the solutions of two extremal problems. As a by-product, the implication `$\Leftarrow$' of \eqref{[-2,2]} (or rather \eqref{iff} below) is obtained. In contrast, the main tool of the present paper is a step-by-step sum rule obtained from a canonical factorization of the $M$-function. This technique was developed by Killip--Simon \cite{MR1999923} and is applied repeatedly in the monograph \cite{Simon}. We shall establish the desired factorization in Sect.~\ref{secM} and arrive at the step-by-step sum rule in \eqref{step-by-step}--\eqref{n step}. In Sect.~\ref{secP}, we show that the equilibrium measure of a Parreau--Widom set is absolutely continuous, provided the set has positive Lebesgue measure. Hence the Szeg\H{o} integral is a relative entropy (up to some constant), and known properties of relative entropy combined with uniform upper bounds on certain eigenvalue sums allow us to iterate the step-by-step sum rule and pass to the limit. With all preparations in place, the proof of Thm.~\ref{main thm} is merely half a page.


Every now and then we shall use the language of Jacobi matrices rather than the one of measures. As is well known, there is a one-one correspondence between compactly supported (nontrivial) probability measures on $\R$ and bounded Jacobi matrices. Given $d\mu$, the associated Jacobi matrix is given by
\[
J=\left( \begin{array}{ccccc}
b_1 & a_1 &  &  & \\
a_1 & b_2 & a_2 & &  \\
& a_2 & b_3 & a_3 & \\
& & \ddots & \ddots & \ddots \\
\end{array} \right),
\]
where $\{a_n, b_n\}_{n=1}^\infty\in (0,\infty)^\N\times\R^\N$ are the recurrence coefficients of the orthonormal polynomials $P_n(x, d\mu)$. When $\supp(d\mu)$ is compact, these coefficients are bounded. The spectrum of $J$, viewed as an operator on $\ell^2(\N)$, coincide with $\supp(d\mu)$ and we shall often refer to $d\mu$ as the spectral measure of $J$.
In the language of Jacobi matrices (and with $g$ the Green's function for $\overline{\C}\setminus\large\E$), the main result of the paper reads:
\begin{thm}
Let $J=\{a_n, b_n\}_{n=1}^\infty$ be a Jacobi matrix with spectral measure $d\mu=f(t)dt+d\mu_\s$
and let $\large\E\subset\R$ be a Parreau--Widom set of $\vert\large\E\vert>0$.
Assume that $\si_\ess(J)=\large\E$ and denote by $\{x_k\}$ the eigenvalues of $J$ outside $\large\E$, if any.
On condition that
\begin{equation}
\label{Blaschke}
\sum_k g(x_k)<\infty,
\vspace{-0.3cm}
\end{equation}
we have
\begin{equation}
\label{iff}
\limsup_{n\to\infty} \frac{a_1\cdots a_n}{\ca(\large\E)^n}>0
\quad\iff\quad
\int_\E\log f(t)\,d\mu_\E(t)>-\infty.
\end{equation}
Moreover, if one and hence both of the equivalent conditions in \eqref{iff} hold true, then
\[
0<\liminf_{n\to\infty}\frac{a_1\cdots a_n}{\ca(\large\E)^n}\leq \limsup_{n\to\infty}\frac{a_1\cdots a_n}{\ca(\large\E)^n}<\infty.
\]
\end{thm}

While the Szeg\H{o} integral relates to relative entropy and allows for defining an outer function (see Sect.~3), the product $a_1\cdots a_n$ is the reciprocal of the leading coefficient in $P_n(x, d\mu)$. So if $\large\E$ is rescaled to have capacity $1$, the leading coefficients in the orthonormal polynomials are bounded above and below.


With \eqref{Sz thm} as a starting point, Hayashi \cite{MR662302} set out to generalize Szeg\H{o}'s theorem to Riemann surfaces. Naturally, such a generalization may not be as clean and explicit as \eqref{[-2,2]} or \eqref{iff}. Interestingly, the results of \cite{MR662302} simplify when the Riemann surface $\RR$ is of Parreau--Widom type. Let $d\omega_\tau$ be the harmonic measure on $\partial\RR$ for a point $\tau\in\RR$. When $h\geq 0$ belongs to $L^1(\partial\RR, d\om_\tau)$, Hayashi proved that
\begin{equation}
\label{RS}
\inf_{f\in H^\infty_\tau(\RR)} \biggl\{\int\vert 1-f\vert^p h\,d\om_\tau\biggr\} \leq
C\exp\biggl\{\int_{\partial\RR}\log h\, d\om_\tau\biggr\},
\end{equation}
where $H_\tau^\infty(\RR)$ is the set of bounded analytic functions on $\RR$ vanishing at $\tau$, and $C>0$ is some constant depending on $\tau$. Provided that $\partial\RR\subset\R$, and if the infimum on the left-hand side can be related to $a_1\cdots a_n/\ca(\partial\RR)^n$, this inequality may prove the implication `$\Rightarrow$' of \eqref{iff} when $d\mu_\s=0$. In order to be able to include a singular part of the measure, \cite{MR662302} needs an extra assumption on $\RR$ which is equivalent to the Direct Cauchy Theorem. We shall not discuss this issue here but refer the reader to \cite{MR723502} and the recent paper \cite{Y}.

\medskip

{\bf Acknowledgements:} The author would like to thank Alexandru Aleman, Katsuhiro Matsuzaki, Barry Simon, Mikhail Sodin, Peter Yuditskii, and Maxim Zinchenko for useful discussions and comments while preparing the paper.

\section{Parreau--Widom sets and uniformization theory}
\label{secPW}

In this section we start by introducing the notion of a Parreau--Widom set on the real line. This notion covers a large class of compact subsets of $\R$ and allows for a set to have infinitely many `gaps', yet no isolated points. The precise definition will be given below. It relies on potential theory and key roles will be played by the Green's function and the equilibrium measure. We refer the reader to \cite{MR2150803,MR1163828,MR0414898,MR694693,MR0279280,MR0350027} for background and more details on potential theory.

In the second part of the section, we give a brief account on uniformization theory. The universal covering map will be brought into play and we relate the Green's function to Blaschke products of the underlying Fuchsian group. See, e.g., \cite{MR1191903, MR0228671} or \cite[Chap.~9]{Simon} for further details.

\subsection{Parreau--Widom sets}
Let $\large\E\subset\R$ be a compact set. We shall always assume that the {logarithmic capacity} of $\large\E$,
denoted $\ca(\large\E)$, is positive so that the domain $\Om=\overline\C\setminus\large\E$ has a \emph{Green's function}.
For fixed $y\in\Om$, we let $g_{\,\Om}(\,\cdot\,, y)$ be the Green's function for $\Om$ with pole at $y$. Recall that
this function is positive and harmonic on $\Om\setminus\{y\}$. Moreover,
\begin{equation}
\label{green y}
g_{\,\Om}(\,\cdot\,, y)+\log\vert\cdot-y\vert
\end{equation}
is harmonic at $y$. The special case $y=\infty$ will be of particular interest to us and we write $g$ instead of $g_{\,\Om}(\,\cdot\,, \infty)$.
Not only is $g(\,\cdot\,)-\log\vert\cdot\vert$ harmonic at $\infty$, but we also have the expansion
\begin{equation}
\label{green-asymp}
g(x)=\log\vert x\vert+\ga({\large\E})+{o}(1),
\end{equation}
where $\ga(\large\E)=-\log\bigl(\ca(\large\E)\bigr)$ is the so-called {Robin's constant} for $\large\E$.

There is a unique probability measure on $\large\E$ of minimal logarithmic energy. This measure is called
the \emph{equilibrium measure} of $\large\E$ and will be denoted $d\mu_{\E}$. It minimizes the integral
\[
I(d\mu)=\int\int\log\frac{1}{\vert s-t\vert}d\mu(t)d\mu(s)
\]
among all probability measures $d\mu$ on $\large\E$ and the minimal energy is given by $I(d\mu_\E)=\ga(\large\E)$. The logarithmic potential of $d\mu_\E$ brings us back to the Green's function through the relation
\begin{equation}
\label{equi-pot}
g(x)=\ga({\large\E})-\int\log\frac{1}{\vert t-x\vert}d\mu_\E(t).
\end{equation}

Besides $\large\E$ having positive logarithmic capacity, we will assume that each point of $\large\E$ is a regular point for $\Om$, that is,
\begin{equation}
\label{reg}
\lim_{\Om\ni x\to t} g(x)=0 \, \mbox{ for all $t\in\large\E$}.
\end{equation}
This in particular means that $\large\E$ has no isolated points as such points are irregular for $\Om$. In short we say that $\large\E$ is regular when \eqref{reg} holds. Equivalent to this is the Green's function being continuous on all of $\C$.

It will often be useful to write $\large\E$ in the form
\begin{equation}
\label{E}
{\large\E}=\bigl[\al, \be\bigr]\setminus\bigcup_j\,\bigl(\al_j, \be_j\bigr),
\end{equation}
where $\cup_j$ is a countable union of disjoint open subintervals of $[\al, \be]$ and $\al<\al_i\neq\be_j<\be$ for all $i, j$.
We shall refer to $(\al_j, \be_j)$ as a `gap' in $\large\E$ and to \eqref{E} as the infinite (or finite) gap representation of $\large\E$. While the Green's function vanishes on $\large\E$ and in particular at the endpoints $\al_i$ and $\be_j$, it is strictly concave on each of the gaps. Since $g$ cannot be constant on any interval in $\R\setminus\large\E$, this follows from \eqref{equi-pot} and the fact that $\log$ is concave. Hence there is a unique point $c_j\in(\al_j, \be_j)$ at which $g$ attains its maximum on $(\al_j, \be_j)$. The $c_j$'s are the \emph{critical points} of $g$ since $g'(c_j)=0$ for each $j$ and $g'$ never vanishes outside $[\al, \be]$.

\begin{definition}
Suppose $\large\E\subset\R$ is a compact set of $\ca(\large\E)>0$ and suppose $\large\E$ is regular. We call $\large\E$ a \emph{Parreau--Widom set} if
\begin{equation}
\label{PW}
\sum_j g(c_j)<\infty,
\end{equation}
where $g$ is the Green's function for $\overline{\C}\setminus\large\E$ with pole at $\infty$ and $\{c_j\}$ are the critical points of $g$.
\end{definition}

\begin{remark}
Note that the definition is independent of which $y\in\Om$ is taken as pole of the Green's function. If the values of $g_{\,\Om}(\,\cdot\,, y)$ at critical points are small enough to be summable for one $y$, the same applies to all $y$. The choice of $y=\infty$ is made for convenience.
\end{remark}

The above terminology is inspired by the monograph \cite{MR723502} of Hasumi. In comparison, we say that $\large\E\subset\R$ is a Parreau--Widom set if the domain $\Om=\overline{\C}\setminus\large\E$ is a Riemann surface of Parreau--Widom type in the language of \cite[Chap.~5]{MR723502}. While originally introduced by Parreau in \cite{MR0098180}, Widom \cite{MR0288780} showed that such surfaces have sufficiently many analytic functions. See, e.g., \cite{MR723502} for more details.

Clearly, any finite gap set (cf. \cite{MR2582999, MR2659747, CSZ2}) is a Parreau--Widom set. But the notion goes way beyond. Jones--Marshall \cite[Sect.~3]{MR827347} proved that it includes infinite gap sets $\large\E$ which are \emph{homogeneous} in the sense of Carleson \cite{MR730079}. By definition, this means there is an $\eps>0$ such that
\begin{equation}
\label{homogeneous}
\frac{\vert (t-\de, t+\de)\cap\large\E\vert}{\de}\geq\eps \,\mbox{ for all $t\in\large\E$ and all $\de<\diam(\large\E)$.}
\end{equation}
Carleson introduced this geometric condition to avoid the possibility of certain parts of $\large\E$ to be very thin, compared to Lebesgue measure. See \cite{MR1057338, MR1022286} for further results on homogeneous sets.
\begin{example}
Remove the middle $1/4$ from the interval $[0,1]$ and continue to remove subintervals of length $1/4^n$ from the middle of each of the $2^{n-1}$ remaining intervals. Let $\E$ be the set of what is left in $[0,1]$; this is a fat Cantor set of $\vert\E\vert=1/2$. One can show that $\vert (t-\de, t+\de)\cap\E\vert\geq\de/4$ for all $t\in\E$ and all $\de<1$.
\end{example}

\subsection{Uniformization theory}
When $\large\E$ has at least one gap, the domain $\Om=\overline{\C}\setminus\large\E$ is not simply connected. So only in the trivial case of $\large\E$ being an interval, $\Om$ is conformally equivalent to the unit disk. For general Parreau--Widom sets, we shall employ uniformization theory as in the seminal paper \cite{MR1674798} of Sodin--Yuditskii. There is a map $\cm: \DD\to\Om$, which is onto but only locally one-to-one, and a Fuchsian group $\Ga$ of M\"{o}bius transformations on $\DD$ so that
\begin{equation}
\label{cover}
\cm(z)=\cm(w) \,\iff\, \exists\ga\in\Ga: \, z=\ga(w).
\end{equation}
This map is called the \emph{universal covering map} and we fix it uniquely by requiring
\begin{equation}
\label{fix x}
\cm(0)=\infty, \quad x_\infty:=\lim_{z\to 0} z\cm(z)>0.
\end{equation}
Note that $\Ga$ is isomorphic to the fundamental group $\pi_1(\Om)$ and hence a free group on as many generators as the number of gaps in $\large\E$.

Since $\ca(\large\E)>0$, it follows from a theorem of Myrberg (see, e.g., \cite[Chap.~XI]{MR0414898}) that $\Ga$ is of convergent type. This means
\[
\sum_{\ga\in\Ga}\bigl(1-\vert\ga(w)\vert\bigr)<\infty \, \mbox{ for all $w\in\DD$},
\]
and hence the Blaschke products defined by
\begin{equation}
\label{B}
B(z,w)=\prod_{\ga\in\Ga}\frac{\vert\ga(w)\vert}{\ga(w)}
\frac{\ga(w)-z}{1-\overline{\ga(w)}z}
\end{equation}
are convergent for $z, w\in\DD$. By convention, a factor in \eqref{B} reduces to $z$ if $\ga(w)=0$. Note that $B(\,\cdot\,, w)$ is analytic on $\DD$ with simple zeros at $\{\ga(w)\}_{\ga\in\Ga}$. The link back to potential theory is given by
\begin{equation}
\label{B and x}
\vert B(z, w)\vert=\exp\bigl\{-g_{\,\Om}\bigl(\cm(z),\, \cm(w)\bigr)\bigr\} \,\mbox{ for $z,w\in\DD$}.
\end{equation}
In particular, the Green's function can be written as
\begin{equation}
\label{B and Green}
g\bigl(\cm(z)\bigr)=-\log\vert B(z)\vert,
\end{equation}
where $B$ is shorthand notation for $B(\,\cdot\,,0)$. We point out that
\begin{equation}
\label{B at 0}
B(z)=\frac{\ca(\large\E)}{x_\infty}z+\mathcal{O}(z^2)
\end{equation}
near $z=0$. Since $B'(0)=\prod_{\ga\neq\id}\vert\ga(0)\vert>0$ and $\cm(z)={x_\infty}/z+\mathcal{O}(1)$ around $z=0$, this follows from \eqref{green-asymp} and \eqref{B and Green}.


\section{A canonical factorization of the $M$-function}
\label{secM}

Let $\large\E\subset\R$ be a Parreau--Widom set and consider a Jacobi matrix $J=\{a_n, b_n\}_{n=1}^\infty$ with $\si_\ess(J)=\large\E$. The spectrum of $J$ thus contains $\large\E$ and consists only of isolated eigenvalues outside $\large\E$. We denote these eigenvalues, if any, by $\{x_k\}$. Should there be infinitely many of them, the $x_k$'s accumulate nowhere but at some (or all) of the endpoints of $\large\E$, viz., $\al, \be$ and $\al_j, \be_j$ in the representation \eqref{E}.

Let $d\mu=f(t)dt+d\mu_\s$ be the spectral measure of $J$ and introduce the $m$-function (or Stieltjes transform of $d\mu$) by
\begin{equation}
\label{m}
m(x):=m_\mu(x)=\int\frac{d\mu(t)}{t-x}, \quad x\in\C\setminus\supp(d\mu).
\end{equation}
It is well known that $m$ is a Nevanlinna--Pick function (i.e., $m$ is analytic in $\C\setminus\R$ and $\im m(x)\gtrless 0$ for $\im x\gtrless 0$). Since $d\mu$ is compactly supported, we readily see that
\begin{equation}
\label{infinity}
m(x)=-{1}/{x}+\mathcal{O}(x^{-2})
\end{equation}
near $\infty$. In fact, one can write down the Laurent expansion of $m_\mu$ around $\infty$ in terms of the moments of $d\mu$. More importantly, the boundary values $m(t+i0):=\lim_{\eps\da 0}m(t+i\eps)$ exist for a.e.\;$t\in\R$ and
\begin{equation}
\frac{1}{\pi}\im m_\mu(t+i\eps)\,dt \xrightarrow[]{\,\,w\,\,} d\mu \,\mbox{ as } \eps\da 0.
\end{equation}
To be even more specific, we have
$f(t)=\frac{1}{\pi}\im m_\mu(t+i0)$ a.e.\;and
\begin{equation}
\label{pp}
\mu_\s\bigl(\{t\}\bigr)=\lim_{\eps\to 0}\eps\im m_\mu(t+i\eps)
\,\mbox{ for all $t\in\R$}.
\end{equation}
The $m$-function remains analytic in the gaps of $\large\E$, and also below $\al$ and above $\be$, except at the eigenvalues $\{x_k\}$ where it has simple poles. Moreover, $m$ is real-valued and strictly increasing on any interval in $\R\setminus\si(J)$ as its derivative is $>0$ there. So the poles and zeros of $m$ interlace on each of the intervals in $\R\setminus\large\E$ and the same applies to $m(x)-a$ for any $a\in\R$.

A major role in what follows will be played by the function
\begin{equation}
\label{M}
M(z):=-m\bigl(\cm(z)\bigr), \quad z\in\DD.
\end{equation}
Here $\cm$ is the covering map defined in \eqref{cover}--\eqref{fix x}.
Compared to $m$, the function $M$ has the advantage of being meromorphic on $\DD$ rather than $\C\setminus\large\E$.
It follows immediately from \eqref{fix x} and \eqref{infinity} that
\begin{equation}
\label{M at infinity}
M(z)=\frac{z}{x_\infty}+\mathcal{O}(z^2)
\end{equation}
near $z=0$.
As a direct consequence of \eqref{cover}, $M$ is \emph{automorphic} with respect to $\Ga$ (i.e., $M\bigl(\ga(z)\bigr)=M(z)$ for every $z\in\DD$ and all $\ga\in\Ga$). The poles of $M$ are situated at the points $p\in\DD$ for which $\cm(p)\in\{x_k\}$. To better keep track of this set, we introduce a fundamental set for $\Ga$ as follows. Consider first the open set
\begin{equation}
\label{Ford}
F:=\bigl\{z\in\DD : \vert\ga'(z)\vert<1 \,\mbox{ for all } \ga\neq\id\bigr\},
\end{equation}
which is known as the Ford fundamental region. Geometrically, $F$ is the unit disk with a number of orthocircles (and their interior) removed. More precisely, one has to remove two orthocircles for each generator $\ga_j$ of $\Ga$ (or one for $\ga_j$ and one for $\ga_j^{-1}$) since the action of $\ga_j$ can be described as inversion in some orthocircle $C_j$ in the upper half-plane following complex conjugation, and $\ga_j^{-1}$ acts similarly with respect to the conjugate circle in the lower half-plane. Besides being symmetric in the real line, the set $F$ has the important properties that
\begin{enumerate}
\item[1)]
no two of its points are equivalent under $\Ga$ (i.e., if $z\in F$ then $\ga(z)\notin F$),
\item[2)]
\label{union}
$\displaystyle{\bigcup_{\ga\in\Ga} \cl\bigl(\ga(F)\bigr)=\DD}$,
where `$\cl$' refers to closure within $\DD$.
\end{enumerate}
We can preserve these two properties and even get a disjoint union in 2) without taking closure by considering
\begin{equation}
\label{F}
\F:=F\cup\bigl(\cup_j C_j\cap\DD\bigr).
\end{equation}
This is a fundamental set for $\Ga$ as it contains one and only one point of each $\Ga$-orbit.

Returning to the poles of $M$, let $p_k$ be the unique point in $\F$ such that $\cm(p_k)=x_k$. Then we can write the collection of poles as $\{\ga(p_k)\}_{k, \ga\in\Ga}$. Note that $p_k$ either belongs to $(-1,1)\setminus\{0\}$ or lies in one of the $C_j$'s, depending on whether $x_k$ is situated outside $[\al, \be]$ or contained in $(\al_j, \be_j)$ for some $j$.
The minus sign on the right-hand side in \eqref{M} ensures that $\im M(z)\gtrless 0$ when $z\in F$ and $\im z\gtrless 0$.

We now aim at establishing an all-important result about $M$, namely that it is a function of {bounded characteristic} on $\DD$ with {no} singular inner part (under a certain condition on the poles $\{x_k\}$, also known as the Blaschke condition). This result can also be found in \cite[Sect.~5]{MR1674798} but we include a complete proof of the statement here, partly due to its importance and partly to make the present paper more self-contained.
\begin{prop}
\label{bnd char}
Let $\large\E$ and $J$ be given as above. In addition, assume that the eigenvalues $\{x_k\}$ satisfy the condition
\begin{equation}
\label{ev}
\sum_k g(x_k)<\infty,
\end{equation}
where $g$ is the Green's function for $\overline{\C}\setminus\large\E$ with pole at $\infty$.
Then the function $M$ defined in \eqref{M} has bounded characteristic.
\end{prop}

\begin{remark}
The condition \eqref{ev} is equivalent to
\begin{equation}
\label{ev y}
\sum_k g_{\,\Om}(x_k,y)<\infty \,\mbox{ for all $y\in\overline{\C}\setminus\si(J)$}.
\end{equation}
For by \eqref{B and Green}, it implies
\[
\prod_{k,\ga\in\Ga}\vert\ga(p_k)\vert=\prod_k\vert B(p_k)\vert>0,
\]
so that $\prod_k B(z,p_k)$ converges to an analytic function on $\DD$ with simple zeros at $\{\ga(p_k)\}_{k,\ga\in\Ga}$. As none of these zeros belong to $\F\setminus\{p_k\}$, the product $\prod_k\vert B(p_k,z)\vert$ is $>0$ there. Hence \eqref{ev y} follows from \eqref{B and x}.
Since every compact set $K\subset\C\setminus\si(J)$ is the image (under $\cm$) of a compact subset of $\F\setminus\{p_k\}$, we have a uniform bound of the form
\begin{equation}
\label{ev uniform}
\sum_k g_{\,\Om}(x_k,y)\leq C \,\mbox{ for all $y\in K$}.
\end{equation}
\end{remark}
Before the proof, let us briefly recall the notion of \emph{bounded characteristic} (see, e.g., \cite{MR0164038} or \cite{MR0279280}). For a meromorphic function $h$, one defines the proximity function by
\[
m(r,h)=\int_0^{2\pi}\log^+\vert h(re^{i\theta})\vert\,\frac{d\theta}{2\pi}
\]
and the counting function by
\[
N(r,h)=\int_0^r \frac{n(t, h)}{t}\,dt,
\]
where $n(t, h)$ is the number of poles of $h$ in $\vert z\vert<t$ (counted with multiplicity). If these poles are denoted $\{p_k\}$, the above integral can also be written as
\begin{equation}
\label{N}
N(r,h)=\sum_k\log\frac{r}{\vert p_k\vert}.
\end{equation}
The sum
\begin{equation}
\label{T}
T(r):=T(r,h)=m(r,h)+N(r,h)
\end{equation}
is called the characteristic function of $h$ and if $\lim_{r\up 1} T(r)<\infty$, we say that $h$ has bounded characteristic in $\vert z\vert<1$.

When $z=0$ is not a pole of $h$, the Cartan identity states that
\[
T(r)=\int_0^{2\pi} N\left(r,\frac{1}{h-e^{i\theta}}\right)\frac{d\theta}{2\pi}+\log^+\vert h(0)\vert.
\]
This formula, obtained by applying Jensen's formula to $h(z)-e^{i\theta}$ and integrating over the unit circle, becomes very useful when the solutions to $h(z)=e^{i\theta}$ are under control for all $\theta$. In a similar way and by use of potential theory, Frostman \cite{Frostman} established the more general estimate
\begin{equation}
\label{Frostman}
T(r)=\int_K N\bigl(r,{1}/({h-a})\bigr)\,d\mu_K(a)+\mathcal{O}(1) \,\mbox{ as $r\to 1$,}
\end{equation}
valid for any set $K\subset\C$ of $\ca(K)>0$ and where $d\mu_K$ is the equilibrium measure of $K$ (see \cite[Chap.~6]{MR0279280}). This estimate in particular tells us that if $N\bigl(r,\frac{1}{h-a}\bigr)\leq C$ for $r<1$ and all $a$ in a set of positive logarithmic capacity, then $T(r)$ is bounded.

\begin{proof}[Proof of Proposition]
The plan is to show that $N\bigl(1,\frac{1}{M-a}\bigr)\leq C$ for all $a$ in some interval of the real line. Given $a\in\R$, let $\{p_k(a)\}$ be the unique points in $\F$ such that $\{\cm(p_k(a))\}$ are the $a$-points of $m$ in $\overline{\R}\setminus\large\E$ (i.e., the solutions of $m(x)=a$). Clearly, the points $x_k(a):=\cm(p_k(a))$ interlace with the poles of $m$ and the collection $\{\ga(p_k(a))\}_{k, \ga\in\Ga}$ accounts for all the $a$-points of $-M$ in $\DD$. Recalling \eqref{N} and \eqref{B and Green}, we have
\[
N\left(1,\frac{1}{M+a}\right)=\sum_{k, \ga\in\Ga}\log\frac{1}{\left\vert\ga\bigl(p_k(a)\bigr)\right\vert}=
-\sum_k\log\left\vert B\bigl(p_k(a)\bigr)\right\vert=\sum_k g\bigl(x_k(a)\bigr).
\]
Hence the task is reduced to dealing with values of the Green's function. Due to interlacing, we immediately see that
\begin{equation}
\label{in gaps}
\sum_{k:\, x_k(a)\in(\al_j,\be_j)}g\bigl(x_k(a)\bigr)\leq g(c_j)+\sum_{k:\, x_k\in(\al_j,\be_j)}g(x_k).
\end{equation}
So the $a$-points in gaps of $\large\E$ do not present any problems. Outside $[\al, \be]$ most of the $a$-points interlace with poles too. But for $a>0$, there may be an $a$-point ($x_{\min}(a)$, say) of $m$ below the smallest eigenvalue of $J$ and for $a<0$, there may be an $a$-point above the largest eigenvalue. However that may be, we always have a pointwise (in $a\neq 0$) estimate of the form
\begin{equation}
\label{above and below}
\sum_{k:\, x_k(a)\in\R\setminus[\al,\be]}g\bigl(x_k(a)\bigr)\leq C(a)+\sum_{k:\, x_k\in\R\setminus[\al,\be]}g(x_k)
\end{equation}
for some constant $C:=C(a)$ depending on $a$. For the choice of $C=g(x_{\min}(1))$, the estimate holds uniformly for $a\geq 1$. Combining \eqref{in gaps} and \eqref{above and below}, we thus arrive at
\[
N\left(1,\frac{1}{M+a}\right)\leq\sum_jg(c_j)+\sum_kg(x_k)+C,
\]
valid for $a\in[1,2]$, say. This completes the proof.
\end{proof}

As a function of bounded characteristic, $M$ has angular boundary values $M(e^{i\theta})$ a.e.\;on the unit circle and admits a factorization of the form
\begin{equation}
\label{factor}
M(z)=\frac{\pi_1(z)}{\pi_2(z)}\exp\biggl\{\int_0^{2\pi}\frac{e^{i\theta}+z}{e^{i\theta}-z}
\log\bigl\vert M(e^{i\theta})\bigr\vert
\frac{d\theta}{2\pi}+\int_0^{2\pi}\frac{e^{i\theta}+z}{e^{i\theta}-z}d\rho(\theta)\biggr\},
\end{equation}
where $\pi_1$, $\pi_2$ are Blaschke products corresponding to zeros and poles of $M$, and $d\rho$ is a singular measure on $\partial\DD$.
Naturally, $M(e^{i\theta})$ coincide a.e.\;with $m(t+i0)$ for suitable $t\in\large\E$. As $e^{i\theta}$ runs through $\partial\F:=\overline{\F}\cap\partial\DD$, the corresponding values of $t$ cover $\large\E$ precisely twice and the same applies when $e^{i\theta}$ traverses $\ga(\overline{\F})\cap\partial\DD$ for arbitrary $\ga\in\Ga$.
With reference to Pommerenke \cite{MR0466534}, see also \cite{MR1126158}, we have
$\sum_{\ga\in\Ga}\vert\ga(\overline{\F})\cap\partial\DD\vert=2\pi$ (since $\large\E$ is a Parreau--Widom set) and
\[
\int_0^{2\pi}h\bigl(\cm(e^{i\theta})\bigr)\frac{d\theta}{2\pi}=
\int_{\partial\F} h\bigl(\cm(e^{i\theta})\bigr)
\sum_{\ga\in\Ga}\bigl\vert\ga'(e^{i\theta})\bigr\vert\frac{d\theta}{2\pi}
\]
whenever $h\circ\cm$ is integrable on $\partial\DD$. Moreover, by preservation of the equilibrium measure under the covering map (see, e.g., \cite[Chap.~3]{MR723502} or \cite[Chap.~2]{MR694693}),
\begin{equation}
\label{preserve equi}
\int_\E h(t)\,d\mu_\E(t)=\int_0^{2\pi}h\bigl(\cm(e^{i\theta})\bigr)\frac{d\theta}{2\pi}
\end{equation}
for every $h\in L^1(\large\E, d\mu_\E)$.

The exponential of the integral in \eqref{factor} that involves $\log\vert M(e^{i\theta})\vert$ is called the \emph{outer part} of $M$ while $\exp$ of the second integral is referred to as the \emph{singular inner part} of $M$. The following result will be crucial to us.
\begin{theorem}
\label{no singular inner part}
In the setting of Prop.~\ref{bnd char}, the function $M$ has no singular inner part. In other words, it can be factorized as
\begin{equation}
\label{fact}
M(z)=B(z)\prod_k\frac{B(z,z_k)}{B(z,p_k)}\exp
\biggl\{\int_0^{2\pi}\frac{e^{i\theta}+z}{e^{i\theta}-z}
\log\bigl\vert M(e^{i\theta})\bigr\vert\frac{d\theta}{2\pi}\biggr\},
\end{equation}
where $z_k$ and $p_k$ belong to the fundamental set $\F$ and are chosen in such a way that $\{\cm(z_k)\}$ and $\{\cm(p_k)\}$ are the
zeros and poles of $m$ in $\R\setminus\large\E$.
\end{theorem}

\begin{proof}
We start by showing that $M-\eps$ has no singular inner part for $\eps>0$. The result will then follow taking $\eps\downarrow 0$.
Given $\eps>0$, write $m_\eps:=m+\eps$ in the form
\[
m_\eps(x)=\vert m_\eps(i)\vert\exp\biggl\{\int_\R\biggl(\frac{1}{t-x}-\frac{t}{t^2+1}\biggr)\xi(t)dt\biggr\},
\]
where $\xi$ is defined a.e.\;on $\R$ by $\xi(t)=\frac{1}{\pi}\Arg m_\eps(t+i0)$. The trick is to split the integral into two parts, namely $i)$ $\int_\E$ and $ii)$ $\int_{\R\setminus\large\E}$. We consider each of the two parts separately.
\newline
\newline
$i) \quad \displaystyle{\check{m}(x)=\exp\biggl\{\int_\E\biggl(\frac{1}{t-x}-\frac{t}{t^2+1}\biggr)\xi(t)dt\biggr\}}$
\newline
\newline
Since $\large\E$ is compact, the behaviour of $m_\E$ is controlled by the function
\[
\varphi(x):=\int_\E\frac{\xi(t)}{t-x}dt.
\]
As a Stieltjes transform, $\varphi$ is holomorphic in $\C\setminus\large\E$ and vanishes at $\infty$. Hence $\varphi\circ\cm$ is analytic on $\DD$ and in order to show that $\exp\{\varphi\circ\cm\}$ has no singular inner part, it suffices to prove that $\varphi\circ\cm$ belongs to the Hardy space $H^1$. For every $f\in H^1$ has a complex Poisson representation of the form
\[
f(z)=i\im f(0)+\int_0^{2\pi}\frac{e^{i\theta}+z}{e^{i\theta}-z}\re f(e^{i\theta})\frac{d\theta}{2\pi},
\quad z\in\DD.
\]
A simple computation shows that
\[
\vert\im\varphi(x)\vert\leq\vert\im x\vert\int_\E\frac{\xi(t)}{\vert x-t\vert^2}dt
\leq\vert\im x\vert\int_\R\frac{dt}{\vert x-t\vert^2}\leq \pi
\]
since $0\leq \xi(t)\leq 1$. So the imaginary part of $\varphi\circ\cm$ is bounded and by M.~Riesz' theorem on conjugate functions (see, e.g., \cite[Chap.~17]{MR924157}), this implies that $\varphi\circ\cm\in H^p$ for all $p<\infty$.
\newline
\newline
$ii) \quad \displaystyle{\hat{m}(x)=\exp\biggl\{\int_{\R\setminus\large\E}\biggl(\frac{1}{t-x}-\frac{t}{t^2+1}\biggr)\xi(t)dt\biggr\}}$
\newline
\newline
Since $m_\eps$ is real-valued (on $\R$) away from $\si(J)$, the function $\xi$ only takes the values $0$ and $1$ in $\R\setminus\large\E$ (except at poles and zeros where it is not defined). More precisely, $\xi=1$ on every interval of the form $(x_k, y_k)$, where $x_k$ is a pole and $y_k$ the following zero or some $\be_j$, whichever comes first. Furthermore, if $\lim_{t\downarrow\al_i}m_\eps(t)<0$ for some $i$, then $\xi=1$ on the interval $(\al_i, y)$, where $y$ is the first zero after $\al_i$ (or some $\be_j$). Otherwise $\xi=0$, and this in particular means that $\xi$ vanishes below $x_-$ (= the minimum of $\al$ and the smallest pole) and above $y_+$ (= the maximum of $\be$ and the largest zero). Hence it suffices to consider the function
\begin{equation}
\label{psi defn}
\psi(x):=\exp\biggl\{\int_{I\setminus\large\E}\frac{\xi(t)}{t-x}dt\biggr\}, \quad I:=[x_-, y_+].
\end{equation}

Let $P:=\{x_k\}$ denote the set of poles of $m$. In case these poles accumulate at all endpoints of $\large\E$, we can write $\psi$ as \begin{equation}
\label{psi}
\psi(x)=\prod_k\frac{x-y_k}{x-x_k},
\end{equation}
where $y_k$ by definition is the first zero to the right of $x_k$. The behaviour of $m$ at $\infty$ ensures that
\begin{enumerate}
\item[1)]
no zero can come before the smallest pole in $(-\infty, \al)$,
\item[2)]
there will always be a zero after the largest pole in $(\be, \infty)$.
\end{enumerate}
The representation in \eqref{psi} remains valid in general if we abuse notation and allow for certain $x_k$'s and $y_k$'s to coincide with suitable $\al_i$'s or $\be_j$'s. In all circumstances, the interval $(x_k, y_k)$ is either contained in a gap of $\large\E$ or in $\R\setminus[\al,\be]$. So each factor in \eqref{psi} -- and hence the full product -- is positive on $\large\E$, except perhaps at certain endpoints where it vanishes or is not defined. Naturally, the $x_k$'s can be ordered and we set
\begin{equation}
\label{psi n}
\psi_n(x)=\prod_{k\leq n}\frac{x-y_k}{x-x_k}.
\end{equation}
Since $\sum_k (y_k-x_k)\leq y_+-x_-$, the finite product in \eqref{psi n} converges uniformly to $\psi(x)$ on compact subsets of $\C\setminus \overline{P}$.

The strategy for showing that $\hat{m}\circ\cm$ has no singular inner part is as follows. We know that $\psi\circ\cm$ admits a factorization of the form \eqref{factor}, involving a ratio of Blaschke products and a singular measure $d\rho$. Move the Blaschke products to the left-hand side, take the logarithm and compare the real-parts. Because of \eqref{B and x}, we get
\begin{multline}
\label{factor psi}
\qquad\log\bigl\vert\psi\bigl(\cm(z)\bigr)\bigr\vert+
\sum_k\Bigl(g_{\,\Om}\bigl(\cm(z), y_k\bigr)-g_{\,\Om}\bigl(\cm(z), x_k\bigr)\Bigr)\\=
\int_0^{2\pi}\frac{1-\vert z\vert^2}{\vert e^{i\theta}-z\vert^2}
\log\bigl\vert\psi\bigl(\cm(e^{i\theta})\bigr)\bigr\vert\,\frac{d\theta}{2\pi}+
\int_0^{2\pi}\frac{1-\vert z\vert^2}{\vert e^{i\theta}-z\vert^2}\,d\rho(\theta), \quad z\in\DD.
\qquad
\end{multline}
The goal is to show that the harmonic function on the left-hand side is the Poisson integral of its boundary values (i.e., the
first integral on the right-hand side since $g_{\,\Om}(\,\cdot\,,y)$ vanishes on $\large\E$ for all $y\in\Om$). This clearly implies $d\rho$ to be the trivial measure. As it seems hard to tell whether the harmonic function in question is the real part of an $H^p$-function for suitable $p>1$, we proceed by approximation.

For $0<\de_1, \de_2<1$, consider the function
\begin{equation}
\label{psi n12}
\phi_{n;\,\de_1,\de_2}(x)=\log\biggl\vert\frac{\de_1+\psi_n(x)}{1+\de_2\psi_n(x)}\biggr\vert+
\sum_{k\leq n}\Bigl( g_{\,\Om}\bigl(x, y_k(\de_1,n)\bigr)-g_{\,\Om}\bigl(x, y_k(1/\de_2,n)\bigr)\Bigr),
\end{equation}
where $\{y_k(\de,n)\}$ are the zeros of $\de+\psi_n(\,\cdot\,)$. It is clear that $y_k(\de,n)\in(x_k,y_k)$ for all $k$, and $y_k(\de_1,n)\nearrow y_k$ as $\de_1\downarrow 0$ while $y_k(1/\de_2,n)\searrow x_k$ as $\de_2\downarrow 0$.
Since
\[
\frac{\de_1+\psi_n(x)}{1+\de_2\psi_n(x)}=\frac{1+\de_1}{1+\de_2}
\prod_{k\leq n}\frac{x-y_k(\de_1,n)}{x-y_k(1/\de_2,n)},
\]
we see from \eqref{green y} that $\phi_{n;\,\de_1,\de_2}$ is harmonic on $\Om$ and continuous throughout $\overline{\C}$. Hence its maximum and minimum is assumed on $\large\E$ ($=$ the boundary of $\Om$). The M\"{o}bius transformation $z\mapsto (\de_1+z)/(1+\de_2z)$ maps $\R_+$ onto the interval $(\de_1, 1/\de_2)$ and since $g_{\,\Om}\bigl(\,\cdot\,,y_k(\de,n)\bigr)$ vanishes on $\large\E$, we therefore have
\begin{equation}
\label{n bound}
\log(\de_1)\leq\phi_{n;\,\de_1,\de_2}(x)\leq\log(1/\de_2), \quad x\in\Om.
\end{equation}

The uniform convergence of $\psi_n$ implies (by Hurwitz's theorem) that $y_k(\de, n)\to y_k(\de)$ as $n\to\infty$, where $\{y_k(\de)\}$ are the zeros of $\de+\psi(\,\cdot\,)$. We claim that $\phi_{n;\,\de_1,\de_2}$ converges locally uniformly on $\Om$ to
\begin{equation}
\label{psi d12}
\phi_{\de_1,\de_2}(x)=\log\biggl\vert\frac{\de_1+\psi(x)}{1+\de_2\psi(x)}\biggr\vert+
\sum_k\Bigl( g_{\,\Om}\bigl(x, y_k(\de_1)\bigr)-g_{\,\Om}\bigl(x, y_k(1/\de_2)\bigr)\Bigr).
\end{equation}
It suffices to consider compact sets $K\subset\Om$ for which the intersection $K\cap(\R\setminus\large\E)$ is a closed interval, $L=[a,b]$ say. There are only finitely many $y_k(\de)$'s in this interval and if none of them are endpoints of $L$ (i.e., $=a$ or $b$), we have precisely the same number of $y_k(\de, n)$'s in $L$ for $n$ large enough (again, by Hurwitz's theorem). Clearly,
\begin{equation}
\label{in L}
\sum_{k:\, y_k(\de, n)\in L}
\Bigl( g_\Om\bigl(x, y_k(\de, n)\bigr)+\log\bigl\vert x-y_k(\de, n)\bigr\vert \Bigr)
\end{equation}
is bounded on $K$, uniformly in $n$, and the claim will follow by dominated convergence if we can find $C>0$ such that
\[
\sum_{k:\, y_k(\de, n)\notin L} g_\Om\bigl(x, y_k(\de, n)\bigr)\leq C
\]
for all $x\in K$ and $n$ sufficiently large. By concavity of the Green's function and with $\{c_{x, j}\}$ the critical points of $g_\Om(\,\cdot\,,x)$, it follows that
\begin{equation}
\label{not in L}
\sum_{k:\, y_k(\de, n)\notin L} g_\Om\bigl(x, y_k(\de, n)\bigr)
\leq g_\Om(a-\eta, x)+g_\Om(b+\eta, x)+\sum_{k:\, x_k\notin L}g_\Om\bigl(x_k, x\bigr)+\sum_{j:\, c_j\notin L}g_\Om\bigl(c_{x, j}, x\bigr)
\end{equation}
for $\eta>0$ small and $n$ sufficiently large. We thus get the desired $C$ on the lines of the remark after Prop.~\ref{bnd char}.

The estimate \eqref{n bound} continues to hold in the limit $n\to\infty$ so that $\phi_{\de_1,\de_2}\circ\cm$ is a bounded harmonic function on $\DD$. Hence it can be written as the Poisson integral of its boundary values, that is,
\begin{equation}
\label{Poisson}
\phi_{\de_1, \de_2}\bigl(\cm(z)\bigr)=
\int_0^{2\pi}\frac{1-\vert z\vert^2}{\vert e^{i\theta}-z\vert^2}
\log\biggl\vert\frac{\de_1+\psi(\cm(e^{i\theta}))}{1+\de_2\psi(\cm(e^{i\theta}))}\biggr\vert
\frac{d\theta}{2\pi}, \quad z\in\DD.
\end{equation}
All that remains is now to let $\de_2\da 0$ and then $\de_1\da 0$ in \eqref{Poisson} and \eqref{psi d12}.
Recalling that $\psi\geq 0$ a.e.\,on $\large\E$, we get by monotone convergence that
\[
\lim_{\de_1\da 0}\lim_{\de_2\da 0}\,\phi_{\de_1, \de_2}\bigl(\cm(z)\bigr)=
\int_0^{2\pi}\frac{1-\vert z\vert^2}{\vert e^{i\theta}-z\vert^2}
\log\bigl\vert\psi\bigl(\cm(e^{i\theta})\bigr)\bigr\vert\frac{d\theta}{2\pi},
\quad z\in\DD.
\]
Since $y_k(\de)$ converges to $y_k$ as $\de\da 0$ and to $x_k$ as $\de\uparrow\infty$, it follows from
\eqref{psi d12} that
\[
\lim_{\de_1\da 0}\lim_{\de_2\da 0}\,\phi_{\de_1, \de_2}(x)=
\log\vert\psi(x)\vert+\sum_k \bigl( g_{\,\Om}(x, y_k)-g_{\,\Om}(x, x_k)\bigr),
\quad x\in\Om
\]
if we use dominated convergence as above. In conclusion, $\psi\circ\cm$ has no singular inner part.

For $\eps>0$ small enough, the largest zero of $m_\eps$ lies to the right of $\beta$ and when $\eps\da 0$, it converges to $\infty$. We therefore get the factor $B(z)$ in \eqref{fact}.
\end{proof}

Along the lines of \cite{MR1999923, MR2083307} we shall now rewrite \eqref{fact} to a nonlocal step-by-step sum rule and
introduce first some notation. Let $J_n$ be the $n$ times stripped Jacobi matrix (i.e., the matrix obtained from $J$ by removing the first $n$ rows and columns) and denote by $d\mu_n=f_n(t)dt+d\mu_{n,\s}$ its spectral measure. In particular, $J_1=\{a_{n+1}, b_{n+1}\}_{n=1}^\infty$ and we let $m_1$ be the associated $m$-function. More generally, $m_n$ denotes the $m$-function for $J_n$ (or Stieltjes transform of $d\mu_n$) and $M_n$ is short for $-m_n\circ\cm$.
Furthermore, we use $\{x_{n,k}\}$ to denote the eigenvalues of $J_n$ (or poles of $m_n$) in $\R\setminus\large\E$ and write $p_{n,k}$ for the points in $\F$ for which $\cm(p_{n,k})=x_{n,k}$.

Related to coefficient stripping is the Stieltjes expansion
\begin{equation}
\label{contfrac}
{m(z)}=\frac{1}{-z+b_1-a_1^2m_1(z)}
\end{equation}
which by iteration leads to a continued fraction representation of $m$. As a direct consequence of \eqref{contfrac}, we see that the zeros of $m$ coincide with the poles of $m_1$. Moreover, taking boundary values of the imaginary parts and recalling that $m(t+i0)\neq 0$ a.e., we get that
\[
\frac{\im m(t+i0)}{\vert m(t+i0)\vert^2}=a_1^2 \im m_1(t+i0)\,\mbox{ for a.e.\;$t\in\R$}.
\]
Pulled back to $\partial\DD$, this means
\[
a_1^2\vert M(e^{i\theta})\vert^2=\frac{\im M(e^{i\theta})}{\im M_1(e^{i\theta})}\,\mbox{ for a.e.\;$\theta$},
\]
provided that $\im M_1(e^{i\theta})\neq 0$ a.e.\;or, equivalently, $\im M(e^{i\theta})\neq 0$ a.e.\;on $\partial\DD$. When the set $\{\theta:\im M(e^{i\theta})\neq 0\}$ has full measure (i.e., $f(t)>0$ for a.e.\;$t\in\large\E$), we can therefore write \eqref{fact} in the form
\begin{equation}
\label{fact M}
{a_1M(z)}=B(z)\prod_k\frac{B(z,p_{1,k})}{B(z,p_k)}
\exp\biggl\{\frac{1}{2}\int_0^{2\pi}\frac{e^{i\theta}+z}{e^{i\theta}-z}
\log\biggl(\frac{\im M(e^{i\theta})}{\im M_1(e^{i\theta})}\biggr)\frac{d\theta}{2\pi}\biggr\}.
\end{equation}
This representation, relating $M$ and $M_1$, provides us with step-by-step sum rules. For our purpose it suffices to compare the constant terms.

If we divide by $B(z)$ in \eqref{fact M} and let $z\to 0$, the left-hand side reduces by \eqref{M at infinity} and \eqref{B at 0} to $a_1/\ca(\large\E)$. According to \eqref{B and Green}, the logarithm of the Blaschke product (over $k$) on the right-hand side simplifies to $\sum_k\bigl(g(x_k)-g(x_{1,k})\bigr)$ and the integral becomes $\int_\E\log(f/f_1)d\mu_\E$, using \eqref{preserve equi}. So we end up with
\begin{equation}
\label{step-by-step}
\log\Bigl(\frac{a_1}{\ca(\large\E)}\Bigr)=
\sum_k\bigl(g(x_k)-g(x_{1,k})\bigr)+
\frac{1}{2}\int_\E\log\biggl(\frac{f(t)}{f_1(t)}\biggr)d\mu_\E(t).
\end{equation}
Iteration now leads to
\begin{equation}
\label{n step}
\log\Bigl(\frac{a_1\cdots a_n}{\ca(\large\E)^n}\Bigr)=
\sum_k\bigl(g(x_k)-g(x_{n,k})\bigr)+
\frac{1}{2}\int_\E\log\biggl(\frac{f(t)}{f_n(t)}\biggr)d\mu_\E(t),
\end{equation}
provided that either $\log f$ or $\log f_n$ is integrable with respect to $d\mu_\E$. The underlying assumption 
\eqref{ev} automatically implies that $\sum_k g(x_{n,k})<\infty$ for all $n$ (cf. Prop.~\ref{bound on ev}).

\section{Preparatory results}
\label{secP}

In this section we present the last results needed to prove our main theorem.   
Throughout the section, $\large\E$ denotes a Parreau--Widom set of positive Lebesgue measure.
We start by showing that the equilibrium measure of $\large\E$ is absolutely continuous with respect to Lebesgue measure. This enables us to relate the Szeg\H{o} integral to relative entropies. Then we establish upper bounds for eigenvalue sums like the one in \eqref{ev}, but now for Jacobi matrices that are different from -- and yet related to -- the original $J$.

\subsection{Absolute continuity of $d\mu_\E$}
We shall prove the following result:
\begin{prop}
\label{equi ac}
The equilibrium measure $d\mu_\E$ is absolutely continuous with respect to Lebesgue measure.
\end{prop}
\begin{remark}
If $\cup_j$ in \eqref{E} is a finite union, then the result is well known (see, e.g., \cite[Chap.~5]{Simon}).
For homogeneous $\large\E$, the result is contained in \cite{MR730079} (see also \cite{MR827347}).
\end{remark}

Needless to say, the statement can only be valid when $\vert\large\E\vert>0$. Our proof relies on two lemmas, inspired by \cite{NVY}, \cite{MR2271928}. While the first applies to any probability measure on $\R$, the second is more specific to equilibrium measure.
\begin{lem}
\label{lem1}
Let $d\rho=w(t)dt+d\rho_\s$ be a probability measure on $\R$ and let
\[
m_\rho(x)=\int_\R\frac{d\rho(t)}{t-x}, \quad x\in\C\setminus\supp(d\rho)
\]
be its Stieltjes transform. If ${m_\rho(t+i0)}/{(t+i)}$ belongs to $L^1(\R)$, then $d\rho$ is absolutely continuous (i.e., $d\rho_\s=0$).
\end{lem}
\begin{proof}
Recall that the boundary values $m_\rho(t+i0)$ exist for a.e.\;$t\in\R$ and that $w(t)=\frac{1}{\pi}\im m_\rho(t+i0)$ a.e.\;on $\R$. Our goal is thus to prove that
\[
m_\rho(x)=\frac{1}{\pi}\int_\R\frac{\im m_\rho(t+i0)}{t-x}\,dt, \quad x\in\C\setminus\R.
\]
By the assumption on ${m_\rho(t+i0)}/{(t+i)}$, the integral
\[
\frac{1}{2\pi i}\int_\R\frac{m_\rho(t+i0)}{t-x}\,dt=:f(x)
\]
defines a holomorphic function in $\C\setminus\R$. We readily see that
\[
f(a+ib)-f(a-ib)=\frac{b}{\pi}\int_\R\frac{m_\rho(t+i0)}{(t-a)^2+b^2}\,dt
\]
and since the right-hand side has the same boundary values as $m_\rho(a+ib)$ for $b\downarrow 0$ (see, e.g., \cite[Thm.~5.30]{MR1307384}),
it follows that
\[
f(x)-f(\bar x)=m_\rho(x)
\]
for $x\in\C_+$. Hence $f(\bar x)$ is holomorphic on $\C_+$, but so is $\overline{f(\bar x)}$, and
therefore $f(\bar x)$ must be constant on $\C_+$. Indeed, $f(\bar x)=0$ for $x\in\C_+$
since $f(-ib)\to 0$ as $b\to\infty$. So we conclude that
\[
f(a+ib)=
\begin{cases}
m_\rho(a+ib)&\mbox{for }\, b>0, \\
\qquad 0&\mbox{for }\, b<0,
\end{cases}
\]
and the goal is now easily achieved noting that
\[
\frac{1}{\pi}\int_\R\frac{\im m_\rho(t+i0)}{t-x}\,dt=
\frac{1}{2\pi i}\int_\R\frac{m_\rho(t+i0)-\overline{m_\rho(t+i0)}}{t-x}\,dt=
f(x)+\overline{f(\bar x)}=m_\rho(x)
\]
for $x\in\C\setminus\R$.
\end{proof}
\begin{lem}
\label{lem2}
Let $K\subset\R$ be a compact set of positive Lebesgue measure and denote by $m_K$ the Stieltjes transform of the equilibrium measure $d\mu_K$. Then the boundary values $m_K(t+i0)$ are purely imaginary for a.e.\;$t\in K$.
\end{lem}
\begin{remark}
In the language of \cite{MR2467016, MR2271928, Remling, MR2504863, PSZ}, the lemma says that $d\mu_K$ (or $m_K$) is \emph{reflectionless} on $K$.  The statement is consistent with the fact that the equilibrium potential is constant a.e.\;on $K$.
\end{remark}

\begin{proof}
Write
\[
K_n=\bigl[\al,\be\bigr]\setminus\bigcup_{j=1}^n\bigl(\al_j,\be_j\bigr),
\quad K:=K_\infty.
\]
It is known that $m_{K_n}$ has the desired property for each $n\in\N$ (see, e.g., \cite[Chap. 5]{Simon}). By passing to a subsequence, if necessary, we can assume that $d\mu_{K_n}\xrightarrow{ w }d\mu_K$ as $n\to\infty$. Hence $m_{K_n}\to m_K$, locally uniformly on ${\C}\setminus K$.

The trick is now to consider the exponential representation of the $m$-function. We have
\begin{equation}
\label{exp}
\frac{m_{K_n}(x)}{\vert m_{K_n}(i)\vert}=\exp\biggl\{\int_\R\biggl(\frac{1}{t-x}-\frac{t}{t^2+1}\biggr)\xi_n(t)dt\biggr\},
\end{equation}
where $\xi_n(t)=\frac{1}{\pi}\Arg m_{K_n}(t+i0)$ a.e.\;on $\R$. Obviously, $\xi_n(t)=0$ for $t<\al$ and $\xi_n(t)=1$ for $t>\be$. So the right-hand side in \eqref{exp} reduces to
\[
\frac{\sqrt{1+\be^2}}{\be-x}\exp\biggl\{\int_\al^\be\biggl(\frac{1}{t-x}-\frac{t}{t^2+1}\biggr)\xi_n(t)dt\biggr\}.
\]
But more importantly, we have $\xi_n(t)=1/2$ for all $t\in K_n$, $n\in\N$.
The uniform convergence of $m_{K_n}$ implies that $\int_\al^\ga\xi_n(t)dt$ converges uniformly for $\ga$ in $[\al,\be]$ (see, e.g., \cite[Sect. 2]{AD}). Equivalently, $\xi_n(t)dt\xrightarrow{ w }\xi(t)dt$ as measures on $[\al, \be]$ (where $\xi$ is short for $\xi_\infty$).
Since the Green's function for $\overline{\C}\setminus K_n$ converges locally uniformly to the Green's function for $\Om:=\overline{\C}\setminus K$, the critical points converge too. Therefore,
\[
\xi_n(t)\xrightarrow[n\to\infty]{\mbox{\small pointwise}}
\begin{cases}
1/2 & \mbox{for }\, t\in K, \\
\,\,\,1 & \mbox{for }\, t\in(\al_j, c_j), \\
\,\,\,0 & \mbox{for }\, t\in(c_j, \be_j),
\end{cases}
\]
where $\{c_j\}$ are the critical points of $g_{\,\Om}$. We conclude that $\xi(t)=1/2$ for a.e.\;$t\in K$ and the result follows.
\end{proof}

\begin{proof}[Proof of Proposition]
By Lemma \ref{lem1}, it suffices to show that $m_\E(t+i0)/(t+i)$ is integrable on $\R$.
We split the integral into 3 parts, namely $i)$ $\int_\E$, $ii)$ $\int_{\cup_j(\al_j,\be_j)}$, and $iii)$ $\int_{\R\setminus[\al, \be]}$. \newline
$i)$ According to Lemma \ref{lem2}, $m_\E(t+i0)$ is purely imaginary a.e.\;on $\large\E$. When restricted to $\large\E$,
$\vert m_\E(t+i0)\vert$ is therefore the absolutely continuous part of a finite measure. Hence,
\[
\int_\E\,\biggl\vert\frac{m_\E(t+i0)}{t+i}\biggr\vert\,dt<\infty.
\]
$ii)$ By \eqref{equi-pot}, we can relate $m_\E$ to the derivative of $g$ in the gaps of $\large\E$ and get the estimate
\[
\int_{\cup_j(\al_j,\be_j)}\biggl\vert\frac{m_\E(t+i0)}{t+i}\biggr\vert\,dt\leq 2\sum_j g(c_j).
\]
Since $\large\E$ is a Parreau--Widom set, the sum on the right-hand side is $<\infty$. \newline
$iii)$ Recall from \eqref{infinity} that $m_\E$ decays like $-1/t$ at $\infty$. Therefore,
\[
\int_{\R\setminus[\al,\be]}\biggl\vert\frac{m_\E(t+i0)}{t+i}\biggr\vert\,dt<\infty.
\]
This completes the proof.
\end{proof}

By Prop.~\ref{equi ac}, we can write the equilibrium measure of $\large\E$ as $d\mu_\E=f_\E(t)dt$.
It follows from \eqref{preserve equi} and Lemma \ref{lem2} that
\begin{equation}
\label{logE}
\int_\E\log f_\E(t)\,d\mu_\E(t)=
\int_0^{2\pi}\log\bigl\vert M_\E(e^{i\theta})\bigr\vert\frac{d\theta}{2\pi}-\log\pi>-\infty.
\end{equation}
Hence the Szeg\H{o} integral (i.e., $\int_\E\log f\,d\mu_\E$) is closely related to a relative entropy, specifically
\begin{equation}
\int_\E\log f(t)\,d\mu_\E(t)-\int_\E\log f_\E(t)\,d\mu_\E(t)=
-\int_\E\log\biggl(\frac{f_\E(t)}{f(t)}\biggr)d\mu_\E(t)=:S(d\mu_\E\,|\,d\mu).
\end{equation}
Moreover, if either $\log f$ or $\log f_n$ belongs to $L^1(\large\E, d\mu_\E)$ then we can write the integral in \eqref{n step} as
\begin{equation}
\label{relativ entropi}
\int_\E \log\biggl(\frac{f(t)}{f_n(t)}\biggr)d\mu_\E(t)=S(d\mu)-S(d\mu_n),
\end{equation}
where $S(\,\cdot\,)$ is short notation for the relative entropy $S(d\mu_\E\,|\,\cdot\,)$.

\subsection{Eigenvalue sums}
Let $P_n$ denote the projection on the subspace spanned by the first $n$ basis vectors in $\ell^2(\N)$ (i.e., the vectors $e_1=(1,0,0,\ldots)$, $e_2=(0,1,0,\ldots)$, etc). The following result applies to Jacobi matrices in general.
\begin{lem}
\label{ev in gap}
Let $J$ be a bounded Jacobi matrix and suppose that $\si_\ess(J)\cap(a,b)=\emptyset$ for some $a<b$. If $J$ has no eigenvalues in $(a,b)$, then both $P_n J P_n$ and $(1-P_n)J(1-P_n)$ have at most one eigenvalue between $a$ and $b$.
\end{lem}

\begin{remark}
Note that $J^{(n)}:=P_nJP_n$ is the upper left $n\times n$ corner of $J$ while $(1-P_n)J(1-P_n)$ is equal to $J_n$, the $n$ times stripped matrix.
\end{remark}

\begin{proof}
By the spectral mapping theorem, $J$ has no eigenvalues in $(a,b)$ if and only if $(J-a)(J-b)\geq 0$. Put differently, this means
\[
\Bigl( J-\frac{a+b}{2}\Bigr)^2\geq \Bigl(\frac{b-a}{2}\Bigr)^2.
\]
Given a projection $P$, write $(PJP)^2=PJ^2P-PJ(1-P)JP$ in order to get
\begin{align*}
\Bigl( PJP-\frac{a+b}{2} \Bigr)^2&=
P\Bigl( J-\frac{a+b}{2} \Bigr)^2P-PJ(1-P)JP \\
&\geq P\Bigl( \frac{b-a}{2} \Bigr)^2P-PJ(1-P)JP,
\end{align*}
where all operators are restricted to $\ran(P)$. When $P=P_n$ or $P=1-P_n$, the perturbation $PJ(1-P)JP$ has rank one and introduces at most one eigenvalue in $(a,b)$.
\end{proof}

In what follows, we specialize to the situation where $J=\{a_n,b_n\}_{n=1}^\infty$ has essential spectrum equal to $\large\E$ (i.e., the setting of Section \ref{secM}). For arbitrary $\tilde{J}$ without eigenvalues, let $\tilde{J}^{(n)}$ be the finite rank perturbation given by
\begin{equation}
\label{J tilde}
\tilde{J}^{(n)}=\left( \begin{matrix}
b_1 & a_1 &&&&& \\
a_1 & b_2 & a_2 &&& \\
\vspace{-0.1cm}
& \ddots & \ddots & \ddots && \\
&& \ddots & b_n & a_n & \\
&&& a_n & \phantom{.}_\ulcorner
& \hspace{-0.15cm} \phantom{.}_{\overline{{\phantom{.}}^{\phantom{\widetilde{\hat J}di}}}}  \\
&&&& {\lvert} & \tilde{J}  \\
\end{matrix} \right).
\end{equation}
Denote by $x_k^{(n)}$ the finitely many eigenvalues of $J^{(n)}$ and by $\tilde{x}_k^{(n)}$ the eigenvalues of $\tilde{J}^{(n)}$.
The following result gives the desired upper bound on eigenvalue sums.
\begin{prop}
\label{bound on ev}
Assume that
\begin{equation}
\label{xk}
\sum_k g(x_k)<\infty.
\end{equation}
Then there is a constant $C>0$ such that
\begin{equation}
\label{bound}
\sum_k g(x_{n,k}), \; \sum_k g\bigl(x_k^{(n)}\bigr), \; \sum_k g\bigl(\tilde{x}_k^{(n)}\bigr) \leq C
\end{equation}
for all $n\geq 1$.
\end{prop}

\begin{proof}
For the first two series, the choice of
\[
C=2\sum_k g(x_k)+\sum_jg(c_j)<\infty
\]
works. This follows immediately from Lemma \ref{ev in gap}. To estimate the third, consider first $J^{(n)}\oplus\tilde{J}$ (i.e., set $a_n=0$ in $\tilde{J}^{(n)}$). Since $\tilde{J}$ has no eigenvalues, this direct sum has the same eigenvalues as $J^{(n)}$. The rank two perturbation coming from $a_n>0$ can be written as
\[
\begin{pmatrix} 0 & a_n \\ a_n & 0\end{pmatrix}=
\frac{a_n}{2}\begin{pmatrix}1&1\\1&1\end{pmatrix}
+\frac{a_n}{2}\begin{pmatrix}-1&\phantom{-}1\\\phantom{-}1&-1\end{pmatrix}.
\]
Near the right end of a gap, the negative rank one perturbation
$\left(\begin{smallmatrix} -1 & \phantom{-}1 \\ \phantom{-}1 & -1
\end{smallmatrix}\right)$ interlaces the eigenvalues and the
positive rank one perturbation $\left(\begin{smallmatrix} 1 & 1 \\
1 & 1 \end{smallmatrix}\right)$ then moves the eigenvalues to the
right. Near the left end of a gap, a similar argument applies. Hence,
\[
\sum_k g\bigl(\tilde{x}_k^{(n)}\bigr)\leq\sum_k g\bigl({x}_k^{(n)}\bigr)+2\sum_jg(c_j)
\]
and the result follows.
\end{proof}
In the next section, we take $\tilde{J}$ to be the Jacobi matrix corresponding to the equilibrium measure of $\large\E$.

\section{Szeg\H{o}'s theorem}
\label{proof}
We are now ready to prove the main result of the paper.
\begin{theorem}
\label{main thm}
Let $J=\{a_n, b_n\}_{n=1}^\infty$ be a Jacobi matrix with spectral measure $d\mu=f(t)dt+d\mu_\s$
and let $\large\E\subset\R$ be a Parreau--Widom set of $\vert\large\E\vert>0$.
Assume that $\si_\ess(J)=\large\E$ and denote by $\{x_k\}$ the eigenvalues of $J$ outside $\large\E$, if any.
On condition that $\sum_k g(x_k)<\infty$, we have
\begin{equation}
\label{Sz}
\int_\E\log f(t)\,d\mu_\E(t)>-\infty
\end{equation}
if and only if
\begin{equation}
\label{limsup}
\limsup_{n\to\infty} \frac{a_1\cdots a_n}{\ca(\large\E)^n}>0.
\end{equation}
In particular, \eqref{Sz} is equivalent to $a_1\cdots a_n/\ca(\large\E)^n\not\rightarrow 0$ for measures $d\mu$ supported on $\large\E$.
\end{theorem}

\begin{proof}
Assume first that the Szeg\H{o} condition \eqref{Sz} holds. Then $S(d\mu)>-\infty$ and by \eqref{n step} together with \eqref{relativ entropi}, we have
\begin{equation}
\label{1-2}
\log\Bigl(\frac{a_1\cdots a_n}{\ca(\large\E)^n}\Bigr)=
\sum_k\bigl(g(x_k)-g(x_{n,k})\bigr)+
\tfrac{1}{2}\bigl(S(d\mu)-S(d\mu_n)\bigr).
\end{equation}
According to Prop.~\ref{bound on ev}, the eigenvalue sum $\sum_k g(x_{n,k})$ is bounded above, uniformly in $n$, and the relative entropy $S(d\mu_n)$ is $\leq 0$ for all $n$. Hence the right-hand side of \eqref{1-2} is bounded below and thus
\begin{equation}
\label{liminf}
\liminf_{n\to\infty}\frac{a_1\cdots a_n}{\ca(\large\E)^n}>0.
\end{equation}
So much the more, \eqref{limsup} is true.

To prove that $\limsup>0$ implies the Szeg\H{o} condition, let $J_\E$ be the Jacobi matrix of $d\mu_\E$ and set $\tilde{J}=J_\E$ in \eqref{J tilde}. Then $\tilde{J}^{(n)}$ reduces to $J_\E$ if we coefficient strip $n$ times. By Prop.~\ref{bound on ev} and because $\log f_\E$ is integrable with respect to $d\mu_\E$, cf.\,\eqref{logE}, the iterated step-by-step sum rule \eqref{n step} applies to $\tilde{J}^{(n)}$.
Since $d\mu_\E$ has no eigenvalues and $S(d\mu_\E)=0$, we arrive at
\begin{equation}
\label{2-1}
\log\Bigl(\frac{a_1\cdots a_n}{\ca(\large\E)^n}\Bigr)=
\sum_k g\bigl(\tilde{x}_k^{(n)}\bigr)+
\tfrac{1}{2}S(d\tilde{\mu}_n),
\end{equation}
where $d\tilde{\mu}_n$ is the spectral measure of $\tilde{J}^{(n)}$. Clearly, $\tilde{J}^{(n)}$ converges strongly to $J$ as $n\to\infty$ and thus $d\tilde{\mu}_n\xrightarrow{w} d\mu$. As relative entropy is weakly upper semi-continuous, we therefore have
\[
\limsup_{n\to\infty} S(d\tilde{\mu}_n)\leq S(d\mu).
\]
Hence,
\begin{equation}
\label{2nd}
\limsup_{n\to\infty}\frac{a_1\cdots a_n}{\ca(\large\E)^n}\leq
C'\exp\bigl\{\tfrac{1}{2}S(d\mu)\bigr\},
\end{equation}
where $C'$ is $\exp$ of the constant in \eqref{bound}. In this way, \eqref{limsup} implies \eqref{Sz}.
\end{proof}

\begin{cor}
In the setting of Thm.~\ref{main thm}, if one and hence both of the equivalent conditions \eqref{Sz}--\eqref{limsup} hold true, then
\[
0<\liminf_{n\to\infty}\frac{a_1\cdots a_n}{\ca(\large\E)^n}\leq \limsup_{n\to\infty}\frac{a_1\cdots a_n}{\ca(\large\E)^n}<\infty.
\]
\end{cor}

\begin{proof}
The statement follows immediately from \eqref{liminf} and \eqref{2nd}.
\end{proof}

\bibliographystyle{plain}
\bibliography{references}

\end{document}